\documentclass[
final
]{dmtcs-episciences}

\usepackage[utf8]{inputenc}
\usepackage{subfigure}

\usepackage{amsmath,amssymb,url,color,tikz,colortbl}
\usetikzlibrary{patterns}
\usepackage[all]{xy}
\usepackage{young}
\usepackage{enumerate}
\usepackage{hyperref}
\usepackage{amsfonts}
\usepackage{amsmath}
\usepackage{amssymb}
\usepackage{amsthm}
\usepackage{xcolor}

\usepackage{listings}
\usepackage{tram}
\usepackage{manfnt}
\usepackage{url,color,tikz,euscript}

\usetikzlibrary{arrows,automata}
\usepackage[colorinlistoftodos,color=red!70]{todonotes}

 \newcommand{\dfn}[1]{\textcolor{blue}{\emph{#1}}}

\theoremstyle{theorem}
\newtheorem{thm}{Theorem}

\newtheorem{lem}[thm]{Lemma}
\newtheorem{lemma}[thm]{Lemma}
\newtheorem{prop}[thm]{Proposition}

\newtheorem{cor}[thm]{Corollary}

\newtheorem{conj}[thm]{Conjecture}

\theoremstyle{definition}
\newtheorem{defn}[thm]{Definition}

\newtheorem{ex}[thm]{Example}
\newtheorem{example}[thm]{Example}

\newtheorem{remark}[thm]{Remark}

\numberwithin{thm}{section}
\catcode`@=11
\def\m@th{\mathsurround\z@}
\def\cases#1{\left\{\,\vcenter{\normalbaselines\m@th
    \ialign{$##\hfil$&\quad##\hfil\crcr#1\crcr}}\right.}
\def\hang{\hangindent 24pt}
\def\d@nger{\medbreak\begingroup\clubpenalty=10000
  \def\par{\endgraf\endgroup\medbreak} %
  \noindent\hang\hangafter=-2
  \hbox to0pt{\hskip-\hangindent\dbend\hfill}}
\outer\def\danger{\d@nger}

\newcommand{\rr}{\mathbb{R}}
\newcommand{\zz}{\mathbb{Z}}

\renewcommand{\ss}{\mathfrak{S}}

\newcommand{\inj}{\operatorname{Inj}}

\newcommand{\sur}{\operatorname{Sur}}

\newcommand{\rg}{\operatorname{RG}}

\newcommand{\rgnc}{\rg_\mathrm{nc}}

\newcommand{\ra}{\rightarrow}
\newcommand{\sm}{\setminus}
\definecolor{green}{HTML}{006600}
\definecolor{orange}{HTML}{FF6200}
\definecolor{purple}{HTML}{990099}
\definecolor{coral}{HTML}{FF7F50}
\definecolor{mahogany}{HTML}{C04000}
\definecolor{gold}{HTML}{DAA541}
\definecolor{chocolate}{HTML}{D5691E}
\newcommand{\co}[2]{{\cellcolor{#1}{\color{white}\textbf{#2}}}}
\newcommand{\cob}[2]{{\cellcolor{#1}\textbf{#2}}}
\newcommand{\stir}[2]{\left\{#1\atop #2\right\}}

\newcommand{\redcir}[1]{\large{{\color{red}\textcircled{\normalsize{{\color{black}#1}}}}}}
\newcommand{\bluecir}[1]{\large{{\color{blue}\textcircled{\normalsize{{\color{black}#1}}}}}}

\catcode`@=12 
\renewcommand{\labelitemii}{\scriptsize\raise2pt\hbox{$\!{\blacktriangleright}$}} 

\newcommand{\eso}{\EuScript{O}}

\newcommand{\calf}{\mathcal{F}}

\newcommand{\call}{\mathcal{L}}

\newcommand{\calp}{\mathcal{P}}

\newcommand{\cals}{\mathcal{S}}

\newcommand{\bfv}{\mathbf{v}}
\newcommand{\bfw}{\mathbf{w}}

\usepackage[round]{natbib}

\author[Michael Joseph et al.]{Michael Joseph\affiliationmark{1}
  \and James Propp\affiliationmark{2}
  \and Tom Roby\affiliationmark{3}}
\title{Whirling injections, surjections, and other functions between finite sets}
\affiliation{
Dalton State College, Dalton, USA\\
University of Massachusetts Lowell, Lowell, USA\\
  University of Connecticut, Storrs, USA}
\begin{document}
\publicationdata{vol. 27:3}{2025}{22}{10.46298/dmtcs.14126}{2024-08-28; 2024-08-28; 2025-09-12; 2025-10-29}{2025-11-02}
\maketitle
\begin{abstract}
 This paper analyzes a certain action called ``whirling'' that can be defined on any family
of functions between two finite sets equipped with a linear (or cyclic) ordering.  Many maps
of interest in dynamical algebraic combinatorics, such as rowmotion of order ideals,  can be
represented as a composition of 
``toggling'' involutions, each of which modifies its object only locally. Similarly whirling
is made up of locally-acting whirling maps which directly generalize toggles, but cycle
through more than two possible outputs.  In this first paper to focus on whirling, we consider it as
a map on subfamilies of functions between finite sets.

For whirling acting on the set of injections or the set of surjections, we prove that within
each whirling orbit, any two elements of the codomain appear as outputs of functions the same
number of times.  This result can be stated in terms of the homomesy phenomenon, which
occurs when a statistic has the same average across every orbit.  We further explore
homomesy results and conjectures for whirling on restricted-growth words, which correspond
to set partitions.  These results extend the collection of combinatorial objects for which
we have interesting dynamics and homomesy, and open the door to considering whirling in other
contexts.
\end{abstract}

\section{Introduction}\label{sec:intro}

In this paper, we explore the orbits associated with an action that can be defined on any
family $\calf$ of functions $f$ from one finite set $S$ to another finite set $T$, where $S$
comes with a linear ordering and $T$ comes with a cyclic ordering. Without loss of
generality we may take $S = [n]:=\{1,2,\dots,n\}$ with its usual ordering and $T = [k]$ with
the mod $k$ cyclic ordering.  The action is generated by an operation $\bfw:\calf\ra\calf$
(``the whirling map'') that is in turn defined as the composition of simpler maps
$w_i:\calf\ra\calf$ (``whirling at $i$'', with $i$ in $[n]$) that repeatedly add $1\bmod k$
to the value of $f$ at $i$ until we get a function in $\calf$.  Our main results are
examples of the \emph{homomesy}\footnote{Greek 
for ``same middle''} \emph{phenomenon}, introduced by the second and third authors \cite{PR15} and defined as follows.

\begin{defn}
Suppose we have a set $\cals$, an invertible map $w:\cals\ra \cals$ such that every $w$-orbit is finite, and a function (``statistic") $f:\cals\ra\rr$.  Then we say the triple $(\cals,w,f)$ exhibits \dfn{homomesy}
if there exists a constant $c\in\rr$ such that for every $w$-orbit $\eso\subseteq \cals$, \[\frac{1}{\#\eso}\sum\limits_{x\in\eso}f(x)=c.\]  In this case, we say that the function $f$ is \dfn{homomesic with average $c$}, or $\dfn{$c$-mesic}$, under the action of $w$ on $\cals$.
\end{defn}

The homomesy phenomenon is quite widespread throughout combinatorial dynamical systems.  See \cite{Rob16} for a survey article.  Since it is a new area of research over the past roughly 15 years, there is still much to learn
about it, including the best techniques for proving homomesy, what types of actions and statistics commonly yield homomesy, and what consequences can arise from homomesy.

Early on, almost all proven homomesy results were for actions whose order 
was straightforwardly computable in general and 
relatively small compared to the size of the ground set $\cals$.  
So it is rarer and notable to find homomesies for actions with unpredictable orbit sizes, as
is the case with the maps in this paper, which also generally have unknown order.  Other such examples can be found in~\cite{EFG+16,ELM+24}.  In some cases, a statistic which is always an integer has a non-integer average across every orbit, which means all of the orbit sizes must be multiples of the constant average's denominator.  This consequence of homomesy appears in Corollary~\ref{cor:orbit-sizes}.

The whirling maps $w_i$ are a generalization of the \emph{toggle} maps that are heavily studied
in dynamical algebraic combinatorics.  Toggles were first introduced by~\cite{CF95}
as a way of writing the rowmotion map as a composition of simpler operations on order
ideals. \cite{Str18} later generalized them to a much wider variety of settings.
There have been many interesting homomesy and periodicity phenomena discovered for actions
defined in terms of toggles; see
e.g.,~\cite{PR15,Rob16,Had21,EFG+16,JR18,ELM+24}.  

In recent work,~\cite{PR24} consider a whirling action on the set of \emph{$k$-bounded
$P$-partitions} of a poset $P$.  Following in the footsteps of~\cite{Had21}, they prove this to be in equivariant bijection with
rowmotion acting on order ideals of the poset $P\times [k]$ for any finite poset $P$.  They then leverage
this to obtain periodicity and homomesy results for rowmotion acting on the poset ${\sf
V}\times [k]$, which has surprisingly good dynamical properties.
A key point here is that beyond the intrinsic interest of whirling as a generalization of
toggling, its study can facilitate our understanding of rowmotion questions that arise independently. 
A related construction is discussed in~\cite{DTW03}; we have borrowed the term ``whirling'' from the paper's notion of ``whirling tours''.

In Section~\ref{sec:injsur} we consider the whirling action on injections and on surjections
from a finite set $[n]$ to another finite set $[k]$.  In both cases, we prove that within each orbit $\eso$, the number $\sum\limits_{f\in\eso} \#f^{-1}(i)$ of times $i\in[k]$ appears as an output of a function in $\eso$ is the same for each $i\in[k]$.  We also prove this for a family of functions that naturally generalizes injections and conjecture it for a similar family that generalizes surjections.  Then in Section~\ref{sec:rg}, we explore the whirling action on various families of restricted growth words, which are functions that encode set partitions.

\section{Whirling on \textit{m}-injections and \textit{m}-surjections between finite sets}\label{sec:injsur}
In this section we first define the families of functions, the $m$-injections and $m$-surjections, that we will consider,
and the whirling action on them.  We also define the basic tools (e.g., orbit boards) that we use
to understand the dynamics.  The next few subsections detail the lemmas and proofs of homomesy.
Then we discuss a consequence of the homomesy for orbit sizes of whirling, for which we have
no other proof.  We finish with a lifting of the dynamics to the piecewise-linear setting,
where the homomesy conjecturally appears to continue to hold.

\subsection{The whirling map on \textit{m}-injections and \textit{m}-surjections}
\begin{defn}
A function $f:S\ra T$ is \dfn{$m$-injective} if every element $t\in T$ appears as an output of $f$ \emph{at most} $m$ times, i.e., $\#f^{-1}(t) \leq m$ for every $t\in T$.  An $m$-injective function is also called an \dfn{$m$-injection}. Injections are the same as 1-injections.
\end{defn}

\begin{defn}
A function $f:S\ra T$ is \dfn{$m$-surjective} if every element $t\in T$ appears as an output of $f$ \emph{at least} $m$ times, i.e., $\#f^{-1}(t) \geq m$ for every $t\in T$.  An $m$-surjective function is also called an \dfn{$m$-surjection}. Surjections are the same as 1-surjections.
\end{defn}

We denote the set of all $m$-injective functions from $[n]$ to $[k]$ by $\inj_m(n,k)$ and the set of all $m$-surjective functions from $[n]$ to $[k]$ by $\sur_m(n,k)$.
We write a function $f:[n]\ra[k]$ in one-line notation as $f(1)f(2)\cdots f(n)$.  For example, the function $f(1)=5$, $f(2)=1$, $f(3)=2$, $f(4)=1$ is written $f=5121$.  (All of our examples will have codomain $[k]$ with $k\leq 9$.)  Note that $m$-injections are not injective in general.  They are called ``width-$m$ restricted functions'' by~\cite{Wal}, who discusses their enumeration.

\begin{defn}
\label{def:whirl-index}
For $f \in \calf \subseteq [k]^{[n]}$ we define a map $w_i:\calf\ra\calf$, called \dfn{whirling
at index $i$}, as follows: repeatedly add 1 (mod $k$) to the value of $f(i)$ until we get a
function in $\calf$.  The new function is $w_i(f)$.    
\end{defn}

Note that we represent our residues (mod
$k$) within $[k]$ rather than $\{0,1,2,\dotsc ,k-1\}$.  The outputs of a function in $\calf$ are always to be considered mod $k$, so we will not continue to write ``mod $k$.''

\begin{example}
Let $\calf=\inj_2(6,4)$ and $f=422343$.  To compute $w_3(f)$, we first add 1 to $f(3)=2$ and get the function $423343$, which is not 2-injective since 3 appears as an output three times.  So, we add 1 again to the third position and get $424343$, which is also not 2-injective.  Adding 1 once more gives a 2-injective function $421343 = w_3(f)$.
\end{example}

\begin{remark}
The map $w_i$ depends heavily on the family of functions $\calf$ on which it acts.  For example, \begin{itemize}
\item if $\calf=\inj_3(7,3)$, then $w_1(1221332)=3221332$,
\item if $\calf=\sur_1(7,3)$, then $w_1(1221332)=2221332$,
\item if $\calf=\sur_2(7,3)$, then $w_1(1221332)=1221332$.
\end{itemize}
Thus, whenever we talk about whirling, we must first make clear what $\calf$ is.  In this
section $\calf$ either refers to $\inj_m(n,k)$ or $\sur_m(n,k)$ for some positive integers
$n,k,m$.  In the next section, we will consider a different family of functions called
\emph{restricted growth words}, which are in bijection with
set partitions.
\end{remark}

\begin{defn}
For a family of functions $\calf$ from $[n]$ to $[k]$, we define the (left-to-right) \dfn{whirling} map $\bfw:\calf\ra\calf$ to be the map that whirls at indices $1,2,\dots,n$ in that order.  So $\bfw=w_n\circ\cdots\circ w_2\circ w_1$.
\end{defn}

\begin{ex}
Let $\calf=\inj_1(4,7)$.  Then
\[
2753\stackrel{w_1}{\longmapsto }
4753\stackrel{w_2}{\longmapsto }
4153\stackrel{w_3}{\longmapsto }
4163\stackrel{w_4}{\longmapsto }
4165
\]
so $\bfw(2753)=4165$.
\end{ex}

\begin{remark}
On any family $\calf$, $w_i$ and $\bfw$ are invertible.  Given $f\in\calf$, we get $w_i^{-1}(f)$ by repeatedly subtracting 1 from the value of $f(i)$ until we get a function in $\calf$.
\end{remark}

We now discuss orbits and homomesy for $\bfw$.  Theorem~\ref{thm:whirlmesy} is the main result of this section.  We will prove the different cases of this theorem in the upcoming subsections~\ref{subsec:whirl-inj},~\ref{subsec:whirl-m-inj}, and~\ref{subsec:whirl-surj}.

\begin{defn}
For $j\in[k]$, define $\eta_j(f)=\#f^{-1}(\{j\})$ to be the number of times $j$ appears as an output of the function $f$.
\end{defn}

\begin{thm}\label{thm:whirlmesy}
Fix $\calf$ to be either $\inj_m(n,k)$ or $\sur_1(n,k)$ for positive integers $n,k,m$.  Then under the action of $\bfw$ on $\calf$, $\eta_j$ is $\frac{n}{k}$-mesic for any $j\in[k]$.
\end{thm}

Note that $\sum_j \eta_j(f) = n$ for all $f$.
It is clear that an equivalent statement to Theorem \ref{thm:whirlmesy} is that $\eta_i-\eta_j$ is 0-mesic for any $i,j\in[k]$.  That is, $i$ and $j$ appear as outputs of functions the same number of times across any orbit.  We also conjecture this result for $m$-surjections in general.

\begin{conj}\label{conj:whirlmesy}
Let $\calf=\sur_m(n,k)$ for positive integers $m,n,k$.  Then under the action of $\bfw$ on $\calf$, $\eta_j$ is $\frac{n}{k}$-mesic for any $j\in[k]$.
\end{conj}

Conjecture~\ref{conj:whirlmesy} has been verified by a computer for all triples $(m,n,k)$ where $n\leq 10$ (for which there are finitely many relevant triples since $\sur_m(n,k)=\varnothing$ if $mk>n$).

\subsection{Orbit board terminology and notation}
In this subsection, we introduce terminology and notation that applies to whirling on any family of functions $\calf$ from $[n]$ to $[k]$ that we will use on various homomesy proofs. We use a generalization of the tuple boards of~\cite{Had21}.

Let $\eso$ be an orbit under the action of $\bfw$ on $\calf$.  We draw a \dfn{board} for the orbit $\eso$ by placing some $f\in\eso$ on the top row.  The function in row $i+1$ is $\bfw^i(f)$ for $i\in[0,\ell(\eso)-1]$, where $\ell(\eso)$ is the length of $\eso$.  For example, in Figure~\ref{fig:orb-inj}, we show a board for an orbit on $\calf=\inj_1(3,6)$.  From this board, we see that if we let $f=621$, then $\bfw(f)=342$, $\bfw^2(f)=563$, and so on, and that $\bfw^{10}(f)=f$.

Notice that within an orbit board, if $f$ is a given row, then the ``partially whirled element'' $(w_i \circ \cdots \circ w_1)(f)$ is given by the first $i$ numbers on the row below $f$ and the last $n-i$ numbers of $f$.  (We say the row ``below'' the bottom row is the top row, as we consider the orbit board to be cylindrical.)  For instance, in the first two rows of Figure~\ref{fig:orb-inj}, $f=621$, $w_1(f)=321$, and $(w_2\circ w_1)(f)=341$.

\begin{figure}[htb]
\begin{center}
\begin{tabular}{ccc}
{6} & {2} & {1} \\
{3} & {4} & {2} \\
{5} & {6} & {3} \\
{1} & {2} & {4} \\
{3} & {5} & {6} \\
{4} & {1} & {2} \\
{5} & {3} & {4} \\
{6} & {5} & {1} \\
{2} & {6} & {3} \\
{4} & {1} & {5} 
\end{tabular}
\end{center}
\caption{The orbit under the action of $\bfw$ on $\inj_1(3,6)$ containing $f=621$.}
\label{fig:orb-inj}
\end{figure}

We use the term \dfn{reading} the orbit board to refer to this action where we start at a certain position $P$ of the board, and continue to the right until we reach the end of the row, and then go to the leftmost position of the row below and continue.
When we refer to a certain position being $x$ positions ``before'' or ``after'' the position $P$, or say the ``previous'' or ``next'' position, we always mean in the reading order.

\begin{defn}
For a given position $P$ in an orbit board, let $(P,h)$ denote the position $h$ places after $P$ in the reading order (or $-h$ places before $P$ if $h<0$).  Let $(P,[a,b])$ denote the $(b-a+1)$-tuple
$\left( (P,a), (P,a+1), \dots, (P,b-1), (P,b) \right)$.
\end{defn}

\begin{example}\label{ex:position}
Consider the following orbit on $\inj_1(4,5)$.  Let $P$ be the position in the second row and second column, shown below surrounded by a black rectangle.  Then $(P,[1,4])$ consists of the four positions circled in red.  Also $(P,[0,4])$ is $(P,[1,4])$ together with $P$, while $(P,[-1,2])$ is the second row.  As the orbit is cylindrical, the bottom right corner is both $(P,14)$ and $(P,-6)$.  Note that $P$ refers only to the \emph{position}, not the value in that position.  So $P\not=(P,5)$ since they are different positions, even though both contain the value 3.  Similarly, we will never write, e.g., $P=3$.

\begin{center}
\begin{tabular}{cccc}
3 & 2 & 1 & 5\\
4 & \co{black}3 & \redcir{2} & \redcir{1}\\
\redcir{5} & \redcir{4} & 3 & 2\\
1 & 5 & 4 & 3\\
2 & 1 & 5 & 4
\end{tabular}
\end{center}
\end{example}

Note that if $P$ is in row $i$ and column $j$, then $(P,[1,n])$ always consists of all positions to the right of $P$ in row $i$, followed with the leftmost $j$ positions of row $i+1$.  Also note that $P([-n,-1])$ consists of the rightmost $n-j+1$ positions of row $i-1$ followed by all $j-1$ positions left of $P$ in row $i$.

\begin{remark}\label{rem:earlier later}
When describing the entire orbit board, it does not make sense to say that a particular position is ``earlier'' or ``later'' than another position, because the orbit board is cylindrical.  However, when working within a tuple of $r$ consecutive positions, these notions are well-defined.  Suppose we are in $(P,[a,b])$ for some position $P$ and $a,b\in\zz$.
Then to say position $Q$ is earlier than position $R$ (equiv.~$R$ is later than
$Q$) within $(P,[a,b])$ means that $Q=(P,c)$ and $R=(P,d)$ for some
$a\leq c< d \leq b$.
Likewise, it makes sense to refer to the last (or earliest) occurrence of 3 within $(P,[a,b])$, but it doesn't make sense to refer to such within the entire orbit board.
\end{remark}

\subsection{The proof of Theorem~\ref{thm:whirlmesy} for injections}\label{subsec:whirl-inj}
In this subsection, we prove Theorem~\ref{thm:whirlmesy} when $\calf=\inj_1(n,k)$.  Even though this is a special case of $m$-injections, the proof here is much simpler.

\begin{lemma}\label{lem:below-inj}
Suppose $i$ is in position $P$ of a board for the $\bfw$-orbit $\eso$ on $\inj_1(n,k)$.
\begin{enumerate}
\item There is exactly one occurrence of $i+1$ within $(P,[1,n])$.
\item There is exactly one occurrence of $i-1$ within $(P,[-n,-1])$.
\end{enumerate}
\end{lemma}

\begin{proof}
To prove (1), suppose position $P$ is in column $j$.  So $f(j)=i$ for some $f\in\eso$.  Then $(P,[1,n])$ contains the multiset of outputs of $(w_j\circ\cdots\circ w_1)(f)$ in some order.  If $(w_{j-1}\circ\cdots\circ w_1)(f)$ does not have $i+1$ already as an output, then by definition $w_j$ changes the output corresponding to the input $j$ from $i$ to $i+1$.  So there is an occurrence of $i+1$ within $(P,[1,n])$. Since $(w_{j}\circ\cdots\circ w_1)(f)$ cannot have any output more than once, this occurrence is unique.

The proof of (2) is analogous to (1) using the inverse of whirling instead.
\end{proof}

\begin{proof}[of Theorem~\ref{thm:whirlmesy} for $\calf=\inj_1(n,k)$]
By Lemma~\ref{lem:below-inj}, we have a bijection between $i$-entries and $(i+1)$-entries for any $i\in[k-1]$, leading immediately to the equidistribution.
\end{proof}

\subsection{The proof of Theorem~\ref{thm:whirlmesy} for \textit{m}-injections}\label{subsec:whirl-m-inj}

Now we prove Theorem~\ref{thm:whirlmesy} for the case $\calf=\inj_m(n,k)$.  We will partition orbit boards into $[k]$-chunks, where each chunk contains $1,2,\dots,k$ and the instance of $i+1$ within a chunk is at most $n$ positions after the instance of $i$ (in the reading order).  
Note that there is not necessarily a \textit{unique} way of partitioning the orbit into chunks; see Figure~\ref{fig:orb-m-inj} for two different ways to partition the orbit of $\inj_2(4,4)$ containing 1441 into $[4]$-chunks.  However, to prove Theorem~\ref{thm:whirlmesy} we simply need to show the existence of at least one partitioning into $[k]$-chunks.

We will again use the notation $(P,h)$ and $(P,[a,b])$ as we did for the injections proof.  We will always refer to the reading order.

\begin{figure}[hb]
\begin{center}
\begin{minipage}{0.45\textwidth}
\centering
\begin{tabular}{cccc}
\co{blue}{1} & \co{magenta}{4} & \co{chocolate}{4} & \co{cyan}{1} \\
\co{blue}{2} & \co{purple}{1} & \co{cyan}{2} & \co{cyan}{3} \\
\co{blue}{3} & \co{purple}{2} & \co{cyan}{4} & \co{blue}{4} \\
\co{brown}{1} & \co{purple}{3} & \cob{pink}{1} & \co{brown}{2} \\
\cob{pink}{2} & \co{purple}{4} & \co{brown}{3} & \cob{pink}{3} \\
\co{brown}{4} & \co{green}{1} & \cob{pink}{4} & \cob{coral}{1} \\
\cob{coral}{2} & \co{green}{2} & \cob{yellow}{1} & \cob{coral}{3} \\
\co{green}{3} & \cob{coral}{4} & \cob{yellow}{2} & \co{green}{4} \\
\co{darkgray}{1} & \cob{gold}{1} & \cob{yellow}{3} & \co{darkgray}{2} \\
\cob{gold}{2} & \co{darkgray}{3} & \cob{yellow}{4} & \cob{gold}{3} \\
\co{darkgray}{4} & \co{orange}{1} & \co{red}{1} & \cob{gold}{4} \\
\co{orange}{2} & \co{red}{2} & \co{orange}{3} & \co{teal}{1} \\
\co{red}{3} & \co{orange}{4} & \co{red}{4} & \co{teal}{2} \\
\co{black}{1} & \cob{lime}{1} & \cob{lime}{2} & \co{teal}{3} \\
\co{black}{2} & \cob{lime}{3} & \co{teal}{4} & \cob{lime}{4} \\
\co{black}{3} & \co{magenta}{1} & \co{chocolate}{1} & \co{magenta}{2} \\
\co{black}{4} & \co{chocolate}{2} & \co{magenta}{3} & \co{chocolate}{3}
\end{tabular}
\end{minipage}
\begin{minipage}{0.45\textwidth}
\centering
\begin{tabular}{cccc}
\co{blue}{1} & \co{magenta}{4} & \co{green}{4} & \co{red}{1} \\
\co{blue}{2} & \cob{gold}{1} & \co{red}{2} & \co{blue}{3} \\
\co{red}{3} & \cob{gold}{2} & \co{blue}{4} & \co{red}{4} \\
\cob{lime}{1} & \cob{gold}{3} & \cob{yellow}{1} & \cob{yellow}{2} \\
\cob{lime}{2} & \cob{gold}{4} & \cob{yellow}{3} & \cob{lime}{3} \\
\cob{yellow}{4} & \co{black}{1} & \cob{lime}{4} & \co{cyan}{1} \\
\co{black}{2} & \co{cyan}{2} & \co{chocolate}{1} & \co{black}{3} \\
\co{cyan}{3} & \co{black}{4} & \co{chocolate}{2} & \co{cyan}{4} \\
\cob{coral}{1} & \cob{pink}{1} & \co{chocolate}{3} & \cob{coral}{2} \\
\cob{pink}{2} & \cob{coral}{3} & \co{chocolate}{4} & \cob{pink}{3} \\
\cob{coral}{4} & \co{teal}{1} & \co{darkgray}{1} & \cob{pink}{4} \\
\co{teal}{2} & \co{darkgray}{2} & \co{teal}{3} & \co{purple}{1} \\
\co{darkgray}{3} & \co{teal}{4} & \co{darkgray}{4} & \co{purple}{2} \\
\co{orange}{1} & \co{brown}{1} & \co{brown}{2} & \co{purple}{3} \\
\co{orange}{2} & \co{brown}{3} & \co{brown}{4} & \co{purple}{4} \\
\co{orange}{3} & \co{magenta}{1} & \co{green}{1} & \co{green}{2} \\
\co{orange}{4} & \co{magenta}{2} & \co{green}{3} & \co{magenta}{3} 
\end{tabular}
\end{minipage}
\end{center}
\caption{Two different ways to partition the same $\bfw$-orbit of $\inj_2(4,4)$ into $[4]$-chunks.}
\label{fig:orb-m-inj}
\end{figure}

\begin{remark}\label{rem:no more than m}
For any position $P$ in column $j$, $(P,[1,n])$ contains the multiset of outputs of $(w_j\circ\cdots\circ w_1)(f)$ for some $f \in \eso$.  Thus, any set of $n$ consecutive positions cannot contain more than $m$ equal values.
\end{remark}

\begin{lemma}\label{lem:below-m-inj}
Suppose $i$ is the value in position $P$ of a board for the $\bfw$-orbit $\eso$ on $\inj_m(n,k)$.
\begin{enumerate}
\item There are at most $m$ occurrences of $i+1$ within $(P,[1,n])$.
\item If the position $(P,n)$ directly below $P$ does not contain $i+1$, then there are exactly $m$ occurrences of $i+1$ within $(P,[1,n-1])$.
\item There are at most $m$ occurrences of $i-1$ within $(P,[-n,-1])$.
\item If the position $(P,-n)$ directly above $P$ does not contain $i-1$, then there are exactly $m$ occurrences of $i-1$ within $(P,[-(n-1),-1])$.
\end{enumerate}
\end{lemma}

\begin{proof}
Suppose position $P$ is in column $j$.  Then $f(j)=i$ for some $f\in\eso$.  By Remark~\ref{rem:no more than m}, there cannot be more than $m$ occurrences of $i+1$, proving (1).

If $(w_{j-1}\circ\cdots\circ w_1)(f)$ does not have $i+1$ already $m$ times as an output, then by definition $w_j$ changes the output corresponding to the input $j$ from $i$ to $i+1$.  This proves (2).

The proofs of (3) and (4) are analogous to (1) and (2) using the inverse of whirling instead.
\end{proof}

\begin{proof}[of Theorem~\ref{thm:whirlmesy} for $\calf=\inj_m(n,k)$]
Within any orbit $\eso$, we will assume without loss of generality that 1 appears as an output at least as often as any other number $2,\dots,k$.  This is because if $i>1$ were to appear as an output more times than $1$, then we could renumber $i,i+1,\dots,k,1,\dots,i-1$ as $1,2,\dots,k$, and the outputs would remain the same relative to each other (mod $k$).

Choose a 1 in the orbit board and call its position $P_1$.  Then by Lemma~\ref{lem:below-m-inj}(2), there exists at least one occurrence of 2 in $(P_1,[1,n])$; pick a position $P_2$ containing such a 2 and place it in the same chunk as $P_1$.  For every $P_i$ containing the value $i$, we select a position $P_{i+1}$ in $(P,[1,n])$ containing the value $i+1$ and place it in the same chunk as $P_i$. Continue this until the chunk contains $1,2,\dots,k$
(that is, until the $k$ positions that comprise the chunk contain the respective values $1,2,\dots,k$).

To start a new chunk, we choose a 1-entry that is not already part of a chunk, and wish to continue the same process.  However, there is not necessarily a unique occurrence of $i+1$ within the next $n$ positions after an $i$, nor a unique occurrence of $i-1$ within the previous $n$ positions before an $i$.  When placing an $(i+1)$-entry in the same chunk as $i$ in position $P$, we want to choose $i+1$ that is not already part of a chunk.  When $(P,[1,n])$ contains such an $(i+1)$-entry, we choose one of them.  However, such an $(i+1)$-entry may not exist depending on how we chose earlier chunks.

See the left side of Figure~\ref{fig:change-chunks} for an example of this problem.  Let $P$ be the position of the 2 in the purple chunk.  Then the only 3 in $(P,[1,n])$ is already part of the brown chunk.  In this case, we reassign the part of the brown chunk starting with 3 to be in the purple chunk.
So the purple chunk is now complete but the brown chunk is incomplete, having just the values 1 and 2.  We now need to complete the brown chunk.

\begin{figure}[htb]
\begin{center}
\begin{minipage}{0.31\textwidth}
\centering
\begin{tabular}{cccc}
\co{blue}{1} & {4} & {4} & {1} \\
\co{blue}{2} & {1} & {2} & \co{blue}{3} \\
{3} & {2} & \co{blue}{4} & {4} \\
{1} & {3} & {1} & {2} \\
{2} & {4} & {3} & {3} \\
{4} & {1} & {4} & {1} \\
{2} & {2} & {1} & {3} \\
{3} & {4} & {2} & {4} \\
{1} & {1} & {3} & {2} \\
{2} & {3} & {4} & {3} \\
{4} & {1} & \co{darkgray}{1} & {4} \\
{2} & \co{darkgray}{2} & {3} & \co{purple}{1} \\
\co{darkgray}{3} & {4} & \co{darkgray}{4} & \co{purple}{2} \\
\textcircled{1} & \co{brown}{\textcircled{1}} & \co{brown}{\textcircled{2}} & \co{brown}{\textcircled{3}} \\
{2} & {3} & {4} & \co{brown}{4} \\
{3} & {1} & {1} & {2} \\
{4} & {2} & {3} & {3} 
\end{tabular}
\end{minipage}
\begin{minipage}{0.31\textwidth}
\centering
\begin{tabular}{cccc}
\co{blue}{1} & {4} & {4} & {1} \\
\co{blue}{2} & {1} & {2} & \co{blue}{3} \\
{3} & {2} & \co{blue}{4} & {4} \\
{1} & {3} & {1} & {2} \\
{2} & {4} & {3} & {3} \\
{4} & {1} & {4} & {1} \\
{2} & {2} & {1} & {3} \\
{3} & {4} & {2} & {4} \\
{1} & {1} & {3} & {2} \\
{2} & {3} & {4} & {3} \\
{4} & {1} & \co{darkgray}{1} & {4} \\
{2} & \co{darkgray}{2} & {3} & \co{purple}{1} \\
\co{darkgray}{3} & {4} & \co{darkgray}{4} & \co{purple}{2} \\
{1} & \co{brown}{1} & \co{brown}{2} & \co{purple}{3} \\
{2} & {3} & {4} & \co{purple}{4} \\
{3} & {1} & {1} & {2} \\
{4} & {2} & {3} & {3} 
\end{tabular}
\end{minipage}
\begin{minipage}{0.31\textwidth}
\centering
\begin{tabular}{cccc}
\co{blue}{1} & {4} & {4} & {1} \\
\co{blue}{2} & {1} & {2} & \co{blue}{3} \\
{3} & {2} & \co{blue}{4} & {4} \\
{1} & {3} & {1} & {2} \\
{2} & {4} & {3} & {3} \\
{4} & {1} & {4} & {1} \\
{2} & {2} & {1} & {3} \\
{3} & {4} & {2} & {4} \\
{1} & {1} & {3} & {2} \\
{2} & {3} & {4} & {3} \\
{4} & {1} & \co{darkgray}{1} & {4} \\
{2} & \co{darkgray}{2} & {3} & \co{purple}{1} \\
\co{darkgray}{3} & {4} & \co{darkgray}{4} & \co{purple}{2} \\
{1} & \co{brown}{1} & \co{brown}{2} & \co{purple}{3} \\
{2} & \co{brown}{3} & \co{brown}{4} & \co{purple}{4} \\
{3} & {1} & {1} & {2} \\
{4} & {2} & {3} & {3} 
\end{tabular}
\end{minipage}
\end{center}
\caption{In the left diagram, we cannot complete the purple chunk because the only position containing 3 within the next $n=4$ positions after the purple chunk's 2 is in the brown chunk already.  Thus, we have to take the end of the brown chunk starting from the 3, place it in the purple chunk (middle), and then complete the brown chunk (right).}
\label{fig:change-chunks}
\end{figure}

\textbf{ Claim: }Suppose for a given $i$ in position $P$, all $(i+1)$-entries within $(P,[1,n])$ are already part of chunks.  Then at least one of these $(i+1)$-entries is in the same chunk as an $i$-entry in $(P,[1,n-1])$.

To explain why the claim is true, we consider two cases.

\textbf{Case 1: The entry directly below position $P$ is $i+1$.}
Let $Q=(P,n)$ be this position below $P$.  Then the $i+1$ in position $Q$ is already in a chunk with an $i$-entry in $(Q,[-n,-1])=(P,[0,n-1])$.  However, we know position $P$ is not already in a chunk with an $(i+1)$-entry.  So the $(i+1)$-entry in position $Q$ must be in the same chunk as an $i$-entry in $(P,[1,n-1])$.

\textbf{Case 2: The entry directly below position $P$ is not $i+1$.}
Then Lemma~\ref{lem:below-m-inj}(2) implies that there are $m$ occurrences of $i+1$ within $(P,[1,n-1])$.  Consider these $(i+1)$-entries.
Each of them is in a chunk with an $i$ that is at most $n$ positions before it
(and at least one position before).
Thus, all of these corresponding $i$-entries are in $(P,[-(n-1),n-2])$.
If they are all in $(P,[-(n-1),-1])$, then there would be $m$ different $i$-entries in $(P,[-(n-1),-1])$.  Adding in the $i$ in position $P$, this
implies there are $m+1$ different $i$-entries in $(P,[-(n-1),0])$, contradicting Remark~\ref{rem:no more than m}.  So at least one of the occurrences of $i+1$ within $(P,[1,n-1])$ is in a chunk with an $i$ in $(P,[1,n-2])$.  Thus, we have proven the claim.

Suppose that we get stuck in partitioning the orbit board into $[k]$-chunks.  In this case, when creating a chunk $C$ through $\{1,\dots,i\}$, we have $i$ in position $P$ and all $(i+1)$-entries in $(P,[1,n])$ are already part of chunks.  Then we will choose a chunk $C'$ for which both the $i$ and $(i+1)$-entries are within $(P,[1,n])$; such a chunk exists by the claim.  We will take the positions of the $i+1,\dots,k$-entries of $C'$ and make them part of the chunk $C$ instead of $C'$.  Then it is just as if $C$ was created previously, and we now must complete $C'$.  This can cause a chain reaction of chunk reassignment and we will now discuss why the process must end.

At any point, let $Q$ be the position of $i+1$ for which we reassign $i+1,\dots,k$ to be in a different chunk.  Then we are always reassigning the $i+1,\dots,k$-entries
to be in the same chunk as an $i$ that is earlier within $(Q,[-n,-1])$ by the previous paragraph.
Therefore, the $i+1$ in position $Q$ can never later be reassigned to the same chunk as the $i$-entry it was previously with.

Given $i\in[k-1]$, there are only finitely many positions with entry $i$ that can be in the same chunk as any specific position with entry $i+1$.
If chunk reassignments continue to be necessary, then eventually for every pair $i\in[k-1]$ and position $Q$ containing $i+1$, $Q$ will be placed in the same chunk as the \emph{earliest} position with $i$ in $(Q,[-n,-1])$, in which case no more
chunk reassignments are possible.
So the process always terminates but one also needs that reassignment is always possible to conclude that the process stops at a full partition of the board into $[k]$-chunks.
\end{proof}

\subsection{The proof of Theorem~\ref{thm:whirlmesy} for surjections}\label{subsec:whirl-surj}
In this section, we prove Theorem~\ref{thm:whirlmesy} in the case where $\calf=\sur_1(n,k)$.  We again use the notation $(P,h)$ and $(P,[a,b])$ as in the previous sections.

\begin{defn}
Call a position $P$ in an orbit board a \dfn{stoplight} if it has the same value as the position below it.
\end{defn}

This term comes from the analogy with traffic lights.  As we look down a column, we always add 1 to the value, except we have to stop when a position is a stoplight.  The stoplights are circled in Figure~\ref{fig:orb-surj}.

\begin{lemma}\label{lem:stoplight}
Let $P$ be a position in an orbit board of $\sur_1(n,k)$, and let $i$ be the value in that position.
\begin{enumerate}
\item If $(P,[1,n-1])$ does not contain $i$ in any position, then the position $(P,n)$ directly below $P$ contains the value $i$.
\item If $(P,[1,n-1])$ contains the value $i$ in some position, then the position $(P,n)$ directly below $P$ contains the value $i+1$.
\end{enumerate}
\end{lemma}

\begin{proof}
Suppose $P$ is in column $j$ and let $f$ be the function on the row with $P$.  So $f(j)=i$.  Then $(P,[0,n-1])$ consists of the multiset of entries of $(w_{j-1}\circ\cdots\circ w_1)(f)$.

If $(P,[1,n-1])$ does not contain $i$ in any position, then $(w_{j-1}\circ\cdots\circ w_1)(f)$ only contains $i$ as an output once.  So to maintain surjectivity, $((w_j\circ\cdots\circ w_1)(f))(j)=i$, and thus $(P,n)$ contains the value $i$.

On the other hand, if $(P,[1,n-1])$ contains $i$ in any position, then $(w_{j-1}\circ\cdots\circ w_1)(f)$ contains $i$ as an output more than once.  Thus applying $w_j$ changes the output corresponding to $j$ from $i$ to $i+1$, since we still get a surjective function.  In this case, $(P,n)$ contains the value $i+1$.
\end{proof}

\begin{figure}[htb]
\begin{center}
\begin{tabular}{cccccccc}
\bluecir{3} & {1} & {1} & \redcir{1} & {4} & {4} & {2} & \bluecir{4}\\
{3} & {2} & \redcir{2} & {1} & {1} & \bluecir{1} & {3} & {4}\\
{4} & \redcir{3} & {2} & {2} & \bluecir{2} & {1} & \redcir{4} & {1}\\
\redcir{1} & {3} & {3} & \bluecir{3} & {2} & {2} & {4} & \redcir{2}\\
{1} & {4} & \bluecir{4} & {3} & {3} & \redcir{3} & {1} & {2}\\
{2} & \bluecir{1} & {4} & {4} & \redcir{4} & {3} & \bluecir{2} & {3}
\end{tabular}
\end{center}
\caption{An example $\bfw$-orbit of $\sur_1(8,4)$ containing $f=31114424$.  For each stoplight $P$ containing $i$, the last position in $(P,[1,n-1])$ that contains $i+1$ (which is a stoplight by Lemma~\ref{lem:next-red-light}) is shown in the same color.}
\label{fig:orb-surj}
\end{figure}

Note that Lemma~\ref{lem:stoplight} implies that a position is a stoplight if and only if it does not contain the same value as any of the next $n-1$ positions.

\begin{lemma}\label{lem:next-red-light}
Let $P$ be a stoplight that contains $i$.  Then there exists at least one position in $(P,[1,n-1])$ that contains $i+1$.  Let $Q$ be the last position in $(P,[1,n-1])$ that contains $i+1$.  Then\begin{enumerate}
\item $Q$ is a stoplight, and
\item $P$ is the last position containing $i$ within $(Q,[-(n-1),-1])$.
\end{enumerate}
\end{lemma}

\begin{proof}
Any set of $n$ consecutive positions $(P,[0,n-1])$  contains the multiset of values of some function
in $\sur_1(n,k)$ and thus $(P,[1,n-1])$ contains a position with entry $i+1$.
To prove~(1), Lemma~\ref{lem:stoplight} assures us that it suffices to prove that $(Q,[1,n-1])$ does not contain $i+1$ in any position.  Every position in $(Q,[1,n-1])$ is either
\begin{itemize}
\item in $(P,[1,n-1])$,
\item the position below $P$ (i.e., $(P,n)$), or
\item the position directly below $R$ for some $R$ (strictly) between $P$ and $Q$.
\end{itemize}
Since $Q$ is the last position in $(P,[1,n-1])$ that contains $i+1$, any position simultaneously in both $(Q,[1,n-1])$ and $(P,[1,n-1])$ cannot contain $i+1$.  The position below $P$ contains $i$ since $P$ is a stoplight.  Let $R$ be a position between $P$ and $Q$.  If the position below $R$ contains $i+1$, then $R$ contains either $i$ or $i+1$.  Since $P$ is a stoplight, no position in $(P,[1,n-1])$ contains $i$, so $R$ cannot contain $i$ (which proves (2)).  If $R$ contains $i+1$, then $R$ is not a stoplight because $Q$ is a position in $(R,[1,n-1])$ with the same value as $R$.  So the position below $R$ cannot contain $i+1$, meaning no positions in $(Q,[1,n-1])$ contain $i+1$.
\end{proof}

\begin{proof}[of Theorem~\ref{thm:whirlmesy} for $\calf=\sur_1(n,k)$]
Within any column in an orbit board, if we skip over the stoplights, then every entry is 1 more than the entry above it. Since the positions in each column that are not stoplights are equidistributed between $1,2,\dots,k$, it suffices to show that the orbit's stoplight entries are equidistributed between $1,2,\dots,k$.

If there are no stoplights in the orbit, then we are done.
If there are stoplights, then Lemma~\ref{lem:next-red-light} describes a bijection between stoplights with entry $i$ and stoplights with entry $i+1$, so the stoplight entries are equidistributed between $1,2,\dots,k$.
\end{proof}

\subsection{Consequences of the homomesy}\label{subsec:consequences}

Let $\calf$ be either $\inj_m(n,k)$ or $\sur_1(n,k)$.  Then given any $j\in[k]$, the number $\eta_j(f)$ of times $j$ appears as an output of the function $f$ is always an integer.  However, the average value of $\eta_j$ across every orbit is not always an integer, which implies a common divisor among all of the orbit sizes.

\begin{cor}\label{cor:orbit-sizes}
Suppose $\calf$ is either $\inj_m(n,k)$ or $\sur_1(n,k)$.  The size of any $\bfw$-orbit of $\calf$ is a multiple of $\frac{k}{\gcd(n,k)}$.
\end{cor}

\begin{proof}
Theorem~\ref{thm:whirlmesy} says that the average value of $\eta_j$ across any $\bfw$-orbit is $\frac{n}{k}$.  When written in lowest terms, the denominator of $\frac{n}{k}$ is $\frac{k}{\gcd(n,k)}$.  Thus, the size of any orbit must be a multiple of $\frac{k}{\gcd(n,k)}$ in order for the average value of $\eta_j$ to be $\frac{n}{k}$.
\end{proof}

Corollary~\ref{cor:orbit-sizes} gives an example where we can use the homomesy to prove a result that neither mentions homomesy nor the statistic that is homomesic.  As the sizes of the orbits are generally unpredictable, this is the only known way to prove this.  This is one of several instances where homomesy has been used to determine divisibility properties of a map; see ~\cite[Corollary 4.8]{EFG+16} for another.

It is straightforward to see that Theorem~\ref{thm:whirlmesy} (and therefore Corollary~\ref{cor:orbit-sizes} also) extends to the case where we replace $\bfw$ by any product of the $w_i$ maps, each used exactly once, in some order.
Let $\pi$ be a permutation on $[n]$, $\bfw_\pi:=w_{\pi(n)}\circ\cdots\circ w_{\pi(2)}\circ w_{\pi(1)}$, and $\calf$ be either $\inj_m(n,k)$ or $\sur_1(n,k)$.  For any $\bfw_\pi$-orbit $\eso=(f_1,f_2,\dots,f_\ell)$, there exists a corresponding $\bfw$-orbit $\eso'=(f_1',f_2',\dots,f_\ell')$ such that $f_i'=f_i\circ\pi$ for all $i$.  Since $\eta_j(f\circ\pi)=\eta_j(f)$, Theorem~\ref{thm:whirlmesy} implies $\eta_j$ is also $\frac{n}{k}$-mesic on orbits of $\bfw_\pi$.  

Functions from a set $E$ to $[2]$ are in a natural correspondence with subsets of $E$ via indicator functions.  In this case, whirling at index $i$ is equivalent to \emph{toggling} at index $i$.  Toggle groups were first introduced by~\cite{CF95} in the specific setting of order ideals of a poset, and generalized more recently by~\cite{Str18} via the definition below.

\begin{defn}[\cite{Str18}]
Let $E$ be a set and $\call\subseteq 2^E$ a set of ``allowed'' subsets of $E$.  Then to every $e\in E$, we define its \dfn{toggle} $t_e:\call\ra\call$ as 
\[t_e(X)=\left\{\begin{array}{ll}
X\cup\{e\} &\text{if $e\not\in X$ and $X\cup\{e\}\in\call$,}\\
X\sm\{e\} &\text{if $e\in X$ and $X\sm\{e\}\in\call$,}\\
X &\text{otherwise.}
\end{array}\right.\]
Each toggle $t_e$ is a permutation on $\call$.  The \dfn{toggle group} is the subgroup of the symmetric group $\ss_\call$ on $\call$ generated by $\{t_e:e\in E\}$.
\end{defn}

Let $n>0$, $0\leq r\leq n/2$ be integers.
For the specific case $\calf=\inj_{n-r}(n,2)$, we can restate our homomesy result in terms of toggle groups.  Note that for $f\in\calf$, the corresponding subset $\{i\in[n] : f(i)=1\}$ has cardinality $\eta_1(f)$.
The $(n-r)$-injective property when $k=2$ is equivalent to $r\leq \eta_1(f) \leq n-r$.

Rephrasing whirling in terms of toggling, the ground set here $E=[n]$ and our set of allowed subsets is $\call_r(n):=\{X\subseteq [n] : r\leq \#X \leq n-r\}$.  Then for each $e\in[n]$, the toggle $t_e:\call_r(n)\ra\call_r(n)$ is
\[t_e(X)=\left\{\begin{array}{ll}
X\cup\{e\} &\text{if $e\not\in X$ and $\#X< n-r$,}\\
X\sm\{e\} &\text{if $e\in X$ and $\#X> r$,}\\
X &\text{otherwise.}
\end{array}\right.\]  This yields the following corollary.

\begin{cor}\label{cor:togglemesy}
Let $\pi$ be a permutation on $[n]$ and $T_\pi:=t_{\pi(n)}\circ\cdots\circ t_{\pi(2)}\circ t_{\pi(1)}$.  Then under the action of $T_\pi$ on $\call_r(n)$, the cardinality statistic is $n/2$-mesic.
\end{cor}

Corollary~\ref{cor:togglemesy} is one of many homomesic statistics in toggle group actions defined as a product of every toggle each used exactly once.  Other such results include rowmotion and promotion on order ideals and antichains of various posets, as well as actions on independent sets of graphs and on noncrossing partitions (including the Kreweras complement), see~\cite{PR15,Rob16,Had21,EFG+16,JR18,ELM+24}.
\subsection{Lifting whirling to the piecewise-linear realm}

We also conjecture a generalization of Corollary~\ref{cor:togglemesy} to the piecewise-linear realm.  Every set $S\in \call_r(n)$ is naturally associated to an indicator vector $(v_1,v_2,\ldots,v_n)\in\{0,1\}^n$ where
$\bfv_i=1$ if $i\in S$ and $\bfv_i=0$ if $i\not\in S$.  For example, the subset $\{2,5,6\}$ of $[7]$ is associated with the vector $(0,1,0,0,1,1,0)$.
Thus, we can think of $\call_r(n)$ as the
set $\{(v_1,v_2,\ldots,v_n)\in\{0,1\}^n : r \leq v_1+v_2+\cdots+v_n\leq n-r\}$.  To get the piecewise-linear version, we now allow the entries of the vector to be real numbers in the interval $[0,1]$.
\begin{defn}
    Let $n>0$ be an integer and $0\leq r \leq \frac{n}{2}$ be a real number.
    Define the polytope
    \[\calp_r(n):=\left\{(v_1,v_2,\ldots,v_n)\in[0,1]^n : r\leq v_1+v_2+\cdots+v_n\leq n-r\right\}.\]
\end{defn}
\begin{defn}
    Let $1\leq i\leq n$.  We define the \dfn{toggle} $t_i:\calp_r(n) \ra \calp_r(n)$ as follows.  Given a vector $\bfv=(v_1,v_2,\ldots,v_n)\in\calp_r(n)$, let
    \begin{align*}
        L &=\max\{0, r-v_1-v_2-\cdots-v_{i-1}-v_{i+1}\cdots-v_n\},\\
        R &=\min\{1, n-r-v_1-v_2-\cdots-v_{i-1}-v_{i+1}\cdots-v_n\}
    \end{align*}
    be the minimum and maximum values that we could change the value of $v_i$ to and still obtain a vector in $\calp_r(n)$.  Then
    \[t_i(\bfv)=(v_1,v_2,\ldots,v_{i-1},L+R-v_i,v_{i+1},\ldots,v_n).\]
\end{defn}

\begin{example}
    Let $n=7, r=2$, and consider the vector
    $(0.9, 0.6, 0.5, 0, 0, 1, 0.4)\in\calp_2(7)$.
    We will apply the toggles $T=t_7\circ \cdots t_2 \circ t_1$ in a left-to-right order.
    To be in $\calp_2(7)$, the sum of the entries must be between 2 and 5.
    We first toggle $t_1$.  The first entry is $0.9$, and this entry can be anything from 0 to 1 and still yield of vector in $\calp_2(7)$.  So $t_1$ changes the first entry to $0+1-0.9=0.1$.
    Now we have $(0.1,0.6,0.5,0,0,1,0.4)$.
    To do the toggle $t_2$, the second entry can be anything from 0 to 1 and still be in $\calp_2(7)$, so we change it from $0.6$ to $0+1-0.6=0.4$, and we now have $(0.1,0.4,0.5,0,0,1,0.4)$.
    We now do the toggle $t_3$.  The third entry is $0.5$, but the sum of the \emph{other} entries is $1.9$, so the third entry must lie in the interval $[0.1,1]$.  Thus $t_3$ changes the third entry to $0.1+1-0.5=0.6$.  Applying $t_4,t_5,t_6,t_7$, we get that
    \[(0.9,0.6,0.5,0,0,1,0.4){\stackrel{T}{\longmapsto}}(0.1,0.4,0.6,1,1,0,0.6).\]
    Applying $T$ repeatedly, we get an orbit of length 1808, and the average sum of the entries is $7/2$ for vectors in the orbit.
\end{example}

We have tested several random orbits and we conjecture that the homomesy in Corollary~\ref{cor:togglemesy} extends
to $\calp_r(n)$.  There is one caveat though.  A vector with irrational entries probably does not always yield a finite orbit.
Thus, we refer to the notion of \dfn{asympotic homomesy}, discussed in~\cite{Vor18,Jos19,DL24}, and which is clearly equivalent to the previous definition of homomesy when all of the orbits are finite.

\begin{defn} {\cite[Definition 5.3.1]{Vor18}}.
Suppose we have a set $\cals$, a map $w:\cals\ra \cals$, and a function (statistic) $f:\cals\ra\rr$.  If there exists $c\in\rr$ such that for every $x\in \cals$,
\[
\lim\limits_{N\ra\infty} \frac{1}{N}\sum\limits_{i=0}^{N-1}f\big(w^i(x)\big)=c,
\]
then we say that $f$ is \dfn{homomesic with average c} (or \dfn{c-mesic}) under the action of $w$ on $\cals$.
\end{defn}

\begin{conj}\label{conj:PL}
    Let $n>0$ be an integer and $0\leq r \leq \frac{n}{2}$ be a real number.
    Let $\pi$ be a permutation on $[n]$ and $T_\pi:=t_{\pi(n)}\circ\cdots\circ t_{\pi(2)}\circ t_{\pi(1)}$.  Then under the action of $T_\pi$ on $\calp_r(n)$, the cardinality statistic is homomesic with average $n/2$.
\end{conj}

It is not uncommon for homomesies in the combinatorial realm to lift to the piecewise-linear realm; see~\cite{EP21,Jos19,HJ22,Hop22}.
However, as far as we know, the only homomesies shown to lift to the piecewise-linear realm are on the order polytope or chain polytope of a poset.  So Conjecture~\ref{conj:PL} is interesting as it would be an example outside of a poset.

\section{Whirling on restricted growth words}\label{sec:rg}
\subsection{Basic definitions, properties, and main result}

A \dfn{partition} of the set $[n]$ is an unordered collection of disjoint subsets of $[n]$, called \dfn{blocks}, whose union is $[n]$.  For example, $\{\{1,2,5\},\{4,7\},\{3,6\}\}$ is a partition of $[7]$.  We often write a partition by writing each block without set braces or commas, and using a bar to separate blocks.  For example, $\{\{1,2,5\},\{4,7\},\{3,6\}\}$ would be written $125|47|36$.  Notice that $125|47|36=125|36|47=63|152|47$ because the order of the blocks is unimportant as well as the order of the numbers within each block.

Another way of encoding a set partition is using its restricted growth word.

\begin{defn}
Let $\pi$ be a partition of $[n]$ with $k$ blocks.  First order the blocks according to their least elements.  The \dfn{restricted growth word} (or \dfn{RG-word}) according to $\pi$ is the function $f:[n]\ra[k]$ where $f(i)=j$ if $i$ is in the $j^\text{th}$ block of $\pi$.
\end{defn}

\begin{example}
For the partition $125|46|37$ of $[7]$, we reorder the blocks as $125|37|46$.  Then its RG-word is $1123132$.
\end{example}

Let $\rg(n,k)$ denote the set of RG-words corresponding to partitions of $[n]$ with exactly $k$ blocks and $\rg(n)$ the set of RG-words corresponding to all partitions of $[n]$ (without specifying the number of blocks).

Let $\stir{n}{k}$ denote the number of partitions of $[n]$ with exactly $k$ blocks and $B_n$ denote the total number of partitions of $[n]$.  These are called the Stirling numbers of the second kind and Bell numbers respectively; e.g., see~\cite[\S1.9]{Sta11}.  RG-words were first introduced by~\cite{Hut63} and have been studied more recently by~\cite{CR17} for their connections to the $q$ and ``$q$-$(1+q)$'' analogues for Stirling numbers of the second kind.

The following proposition is clear from the way RG-words are defined, since the blocks are ordered via least elements.

\begin{prop}\label{prop:rg-criteria}\hspace{-3 in}{\tiny .}

\begin{itemize}
\item A function $f:[n]\ra[k]$ is in $\rg(n,k)$ if and only if it is surjective and for all $1\leq j\leq k-1$, $\min f^{-1}(j) < \min f^{-1}(j+1)$.
\item A function $f:[n]\ra[n]$ is in $\rg(n)$ if and only if for any $1\leq j\leq n-1$,
\begin{itemize}
\item if $j+1$ is in the range of $f$, then so is $j$, and in this case,
\item $\min f^{-1}(j) < \min f^{-1}(j+1)$.
\end{itemize}
\end{itemize}
\end{prop}

Note that we consider the codomain of functions in $\rg(n,k)$ to be $[k]$ and  of $\rg(n)$ to be $[n]$.  Therefore for $\calf=\rg(n,k)$, $w_i$ adds $1\bmod k$ repeatedly to the value $f(i)$ until we get a function in $\calf$, but for $\calf=\rg(n)$, $w_i$ adds $1\bmod n$ repeatedly instead.

\begin{prop}\label{prop:whirl-rg}
Let $\calf=\rg(n,k)$, $f\in \calf$, and $i\in[n]$.  Let $f(i)=j$.  Then $w_i$ changes the output at $i$ in the following way.

\begin{itemize}
\item If $j\not=k$, then
\[(w_i(f))(i)=\left\{\begin{array}{ll}
j+1 &\text{if there exists $i'<i$ such that $f(i')=j$,}\\
j &\text{if the only $i'<\min f^{-1}(j+1)$ s.t. $f(i')=j$ is $i'=i$,}\\
1 &\text{otherwise.}
\end{array}\right.\]
\item If $j=k$, then 
\[(w_i(f))(i)=\left\{\begin{array}{ll}
k &\text{if the only value $i'$ for which $f(i')=k$ is $i'=i$,}\\
1 &\text{otherwise.}
\end{array}\right.\]
\end{itemize}
\end{prop}

Refer to Example~\ref{ex:RG1} for an example of Proposition~\ref{prop:whirl-rg}.

\begin{proof}
\textbf{Case 1: $f(i)\not=k$.} We have two subcases to consider.\\
\textbf{Case 1a: There exists $i'<i$ such that $f(i')=j$.} Then changing the value of $f(i)$ will not change the RG-word criterion $\min f^{-1}(j) < \min f^{-1}(j+1)$ from Proposition~\ref{prop:rg-criteria}.  So $(w_i(f))(i)=j+1$.\\
\textbf{Case 1b: There does not exist $i'<i$ such that $f(i')=j$.} Then we cannot change the value of $f(i)$ to anything larger than $j$ and still have an RG-word.  If the only $i'<\min f^{-1}(j+1)$ such that $f(i')=j$ is $i'=i$, then changing the value of $f(i)$ to something other than $j$ will violate $\min f^{-1}(j) < \min f^{-1}(j+1)$ and not be an RG-word.  Otherwise, changing the value of $f(i)$ to 1 results in an RG-word.\\
\textbf{Case 2: $f(i)=k$.} Since an RG-word must be surjective, if there does not exist $i'\not=i$ satisfying $f(i')=k$, then we cannot change the value of $f(i)$ to something other than $k$.  So $(w_i(f))(i)=k$.  Otherwise, changing the value of $f(i)$ to 1 still results in an RG-word, and so $(w_i(f))(i)=1$.
\end{proof}

For the $\calf=\rg(n)$ case, we have the following slightly different characterization of $w_i$.  We omit the proof as it is straightforward and similar to the proof of Proposition~\ref{prop:whirl-rg}.

\begin{prop}\label{prop:whirl-rg-n}
Let $\calf=\rg(n)$, $f\in \calf$, and $i\in[n]$.  Let $f(i)=j$.  Then 
\[(w_i(f))(i)=\left\{\begin{array}{ll}
j+1 &\text{if there exists $i'<i$ such that $f(i')=j$,}\\
j &\text{if $f^{-1}(j+1)\not=\varnothing$ and the only $i'<\min f^{-1}(j+1)$ s.t. $f(i')=j$ is $i'=i$,}\\
1 &\text{otherwise.}
\end{array}\right.\]
\end{prop}

\begin{ex}\label{ex:RG1}
For $\calf=\rg(7,4)$,
\[
1213341\stackrel{w_1}{\mapsto}
1213341\stackrel{w_2}{\mapsto}
1213341\stackrel{w_3}{\mapsto}
1223341\stackrel{w_4}{\mapsto}
1221341\stackrel{w_5}{\mapsto}
1221341\stackrel{w_6}{\mapsto}
1221341\stackrel{w_7}{\mapsto}
1221342
\]
so $\bfw(1213341)=1221342$.

On the other hand, for $\calf=\rg(7)$,
\[
1213341\stackrel{w_1}{\mapsto}
1213341\stackrel{w_2}{\mapsto}
1213341\stackrel{w_3}{\mapsto}
1223341\stackrel{w_4}{\mapsto}
1221341\stackrel{w_5}{\mapsto}
1221341\stackrel{w_6}{\mapsto}
1221311\stackrel{w_7}{\mapsto}
1221312
\]
so $\bfw(1213341)=1221312$.
\end{ex}

For $\calf=\rg(n,k)$ or $\calf=\rg(n)$, $w_1$ acts trivially since any RG-word $f$ satisfies $f(1)=1$.  Thus, $\bfw=w_n\circ\cdots w_3\circ w_2$ on these families of functions.  So when we consider generalized toggle orders we define $\bfw_\pi=w_{\pi(n)}\circ\cdots\circ w_{\pi(3)}\circ w_{\pi(2)}$ where $\pi$ is a permutation of $\{2,3,\dots,n\}$.

\begin{defn}
Let \[I_{i\mapsto 1}(f)=\left\{\begin{array}{ll}
1 &\text{if }f(i)=1,\\
0 &\text{if }f(i)\not=1.
\end{array}\right.\]
\end{defn}

The main homomesy result is the following.

\begin{thm}\label{thm:whirl-rg}
Let $n\geq 2$.  Fix $\calf$ to be either $\rg(n)$ or $\rg(n,k)$ for some $1\leq k\leq n$.  Let $\pi$ be a permutation of $\{2,3,\dots,n\}$.  Under the action of $\bfw_\pi$ on $\calf$, $I_{i\mapsto 1}-I_{j\mapsto 1}$ is $0$-mesic for any $i,j\in\{2,3,\dots,n\}$.
\end{thm}

See Figure~\ref{fig:rg(5,3)-orbits} for an illustration of Theorem~\ref{thm:whirl-rg} for the orbits under the action of $\bfw$ (i.e., $\pi$ is the identity) over $\calf=\rg(5,3)$.  This theorem is another instance where there is a homomesy under an action that produces unpredictable orbit sizes, and for which the order of the map is unknown in general.

\begin{figure}[htb]
\begin{center}
\begin{minipage}{.35\textwidth}
\centering
\begin{tabular}{c|c}
$f$            & 12123\\\hline\vspace{-1.05 em}\\
$\bfw(f)$      & 11231\\\hline\vspace{-1.05 em}\\
$\bfw^{2}(f)$  & 12312\\\hline\vspace{-1.05 em}\\
$\bfw^{3}(f)$  & 12323\\\hline\vspace{-1.05 em}\\
$\bfw^{4}(f)$  & 12131\\\hline\vspace{-1.05 em}\\
$\bfw^{5}(f)$  & 12232\\\hline\vspace{-1.05 em}\\
$\bfw^{6}(f)$  & 11233\\\hline\vspace{-1.05 em}\\
$\bfw^{7}(f)$  & 12311\\\hline\vspace{-1.05 em}\\
$\bfw^{8}(f)$  & 12322\\\hline\vspace{-1.05 em}\\
$\bfw^{9}(f)$  & 12333\\\hline\vspace{-1.05 em}\\
$\bfw^{10}(f)$ & 12113\\\hline\vspace{-1.05 em}\\
$\bfw^{11}(f)$ & 12223\\\hline\vspace{-1.05 em}\\
$\bfw^{12}(f)$ & 11123\\\hline\vspace{-1.05 em}\\
$\bfw^{13}(f)$ & 12231\\\hline\vspace{-1.05 em}\\
$\bfw^{14}(f)$ & 11232\\\hline\vspace{-1.05 em}\\
$\bfw^{15}(f)$ & 12313
\end{tabular}
\end{minipage}
\begin{minipage}{.35\textwidth}
\centering
\begin{tabular}{c|c}
$g$            & 12213\\\hline\vspace{-1.05 em}\\
$\bfw(g)$      & 11223\\\hline\vspace{-1.05 em}\\
$\bfw^{2}(g)$  & 12331\\\hline\vspace{-1.05 em}\\
$\bfw^{3}(g)$  & 12132\\\hline\vspace{-1.05 em}\\
$\bfw^{4}(g)$  & 12233\\\hline\vspace{-1.05 em}\\
$\bfw^{5}(g)$  & 11213\\\hline\vspace{-1.05 em}\\
$\bfw^{6}(g)$  & 12321\\\hline\vspace{-1.05 em}\\
$\bfw^{7}(g)$  & 12332\\\hline\vspace{-1.05 em}\\
$\bfw^{8}(g)$  & 12133
\end{tabular}
\end{minipage}
\end{center}
\caption{The two $\bfw$-orbits for $\calf=\rg(5,3)$.  The left orbit has length 16 and the right orbit has length 9.  Notice that in the left orbit, there are the same numbers (four each) of functions $h$ satisfying any of the conditions $h(2)=1$, $h(3)=1$, $h(4)=1$, or $h(5)=1$.  In the right orbit, there are also the same numbers (two each) of functions $h$ satisfying any of the conditions $h(2)=1$, $h(3)=1$, $h(4)=1$, or $h(5)=1$.  This illustrates Theorem~\ref{thm:whirl-rg} for the case $n=5$, $k=3$, and $\pi$ is the identity.}
\label{fig:rg(5,3)-orbits}
\end{figure}

For RG-words, results for $\bfw$ need not necessarily extend to other $\bfw_\pi$ products like they did for $\inj_m(n,k)$ and $\sur_m(n,k)$, since a rearrangement of an RG-word is not always an RG-word.  In fact, these other whirling orders do not always yield the same orbit structure as $\bfw$.  Therefore, we prove this result keeping an arbitrary $\pi$ in mind.
\subsection{Proof of homomesy for whirling restricted growth words}
We will use several lemmas in the proof of Theorem~\ref{thm:whirl-rg}.  For the rest of this section, assume $n\geq 2$.

\begin{lem}\label{lem:rg-change-order}
Let $\calf$ be either $\rg(n)$ or $\rg(n,k)$ for some $1\leq k\leq n$.  Let $\pi,\sigma$ be permutations on $\{2,3,\dots,n\}$ where $\sigma(i)=\pi(i+1)$ for $2\leq i\leq n-1$ and $\sigma(n)=\pi(2)$.  Then for any $\bfw_\pi$-orbit $\eso=(f_1,f_2,\dots,f_\ell)$, there is a $\bfw_\sigma$-orbit $\eso'=(f_1',f_2',\dots,f_\ell')$ for which $f_j'=w_{\pi(2)}(f_j)$ for all $j\in [\ell]$.

Also, given $h\in[n]$, the multiset of values $f(h)$ as $f$ ranges over $\eso$ is the same as that for $f$ ranging over $\eso'$.
\end{lem}

\begin{proof}
Let $\eso=(f_1,f_2,\dots,f_\ell)$ be a $\bfw_\pi$-orbit satisfying $\bfw_\pi(f_j) = f_{j+1}$ for all $j$.  Consider the subscripts to be mod $\ell$, the length of the orbit, so $f_{\ell+1}=f_1$ for instance.  Note that $\bfw_\pi=w_{\pi(n)}\circ\cdots\circ w_{\pi(3)}\circ w_{\pi(2)}$ and $\bfw_\sigma= w_{\pi(2)} \circ w_{\pi(n)}\circ \cdots\circ w_{\pi(3)}$.  Therefore, if we let $f_j'=w_{\pi(2)}(f_j)$ for all $j$, then $\bfw_\sigma(f_j') = f_{j+1}'$ for all $j$.  So $\eso'=(f_1',f_2',\dots,f_\ell')$ is a $\bfw_\sigma$-orbit.

For $h\not=\pi(2)$, $f_i(h)=f_i'(h)$.  Also, $f_i'(\pi(2))=f_{i+1}(\pi(2))$.  Thus for any $h\in[n]$, the multiset of values of $f(h)$ is the same for $f$ ranging over $\eso$ is the same as that for $\eso'$.
\end{proof}

\begin{remark}\label{rem:rg-change-order}
Suppose we have $\bfw_\pi = w_{\pi(n)}\circ\cdots\circ w_{\pi(3)}\circ w_{\pi(2)}$ and $\bfw_\sigma$ is some cyclic shift of this order of composition.  Then by applying Lemma~\ref{lem:rg-change-order} repeatedly, we get that every $\bfw_\pi$-orbit $\eso$ corresponds uniquely with a $\bfw_\sigma$-orbit $\eso'$ where the multisets of $f(h)$ values are the same for $f$ ranging over $\eso$ as for $\eso'$.

For example, on $\calf=\rg(6)$, $w_3\circ w_5\circ w_4\circ w_6\circ w_2$ and $w_4\circ w_6\circ w_2 \circ w_3\circ w_5$ satisfy the orbit correspondence above. However, in general two orders for applying the whirling maps do not necessarily yield the same orbit structure if they are not \emph{cyclic} rotations of each other.  For example, $w_3\circ w_5\circ w_4\circ w_6\circ w_2$ has an orbit structure different from that of $w_5\circ w_6\circ w_3\circ w_2\circ w_4$.  The former has orbit sizes $5,6,6,23,35,62,66$ while the latter has orbit sizes $4,6,20,34,139$.
\end{remark}

\begin{lemma}\label{lem:sphynxinator}
Let $f\in\calf$ where $\calf$ is either $\rg(n)$ or $\rg(n,k)$ for some $1\leq k\leq n$.  Let $\pi$ be a permutation of $\{2,3,\dots,n\}$.  Let $i=\pi(2)$ and suppose $2\leq i\leq n-1$.  If
$\left(\bfw_\pi^{-1}(f)\right)(i) < \left(\bfw_\pi^{-1}(f)\right)(i+1)$, then $f(i)\not=1$.
\end{lemma}

\begin{proof}
Let $g=\bfw_\pi^{-1}(f)$ and suppose $g(i)<g(i+1)$. Let $g(i)=j$.
Since $g(i)<g(i+1)$, $g^{-1}(j+1)\not=\varnothing$.  By Propositions~\ref{prop:whirl-rg} and~\ref{prop:whirl-rg-n}, $f(i)=(\bfw(g))(i)=(w_i(g))(i)\not=1$.
\end{proof}

\begin{lemma}\label{lem:global-avg}
Let $i\in\{2,3,\dots,n\}$.
\begin{enumerate}
\item The number of $f\in\rg(n,k)$ satisfying $f(i)=1$ is $\stir{n-1}{k}$.
\item The number of $f\in\rg(n)$ satisfying $f(i)=1$ is $B_{n-1}$.
\end{enumerate}
\end{lemma}

\begin{proof}
The condition $f(i)=1$ means $1$ and $i$ are in the same block of the partition of $[n]$ corresponding to $f$.  Such set partitions can be formed by first choosing a partition of $\{2,3,\dots,n\}$ and placing 1 in the same block as $i$.
\end{proof}

In particular, the counting formulas in Lemma~\ref{lem:global-avg} do not depend on $i$, which is what is important for the proof of Theorem~\ref{thm:whirl-rg}.

\begin{proof}[of Theorem~\ref{thm:whirl-rg}]
Let $\eso$ be a $\bfw_\pi$-orbit.  The goal is to prove that in $\eso$, there are the same number of functions $f\in\eso$ satisfying $f(i)=1$ as there are satisfying $f(j)=1$, when $2\leq i,j\leq n$.  We will approach this by showing that $\eso$ has the same number of functions $f$ satisfying $f(i)=1$ as $f(i+1)=1$, for $2\leq i\leq n-1$.

Let $i\in\{2,3,\dots,n-1\}$.  By Remark~\ref{rem:rg-change-order}, we may assume without loss of generality that $i=\pi(2)$.

Let $f\in\eso$ be a function satisfying $f(i)=1$ and $f(i+1)\not=1$.  By Lemma~\ref{lem:sphynxinator}, $\left(\bfw_\pi^{-1}(f)\right)(i)\geq \left(\bfw_\pi^{-1}(f)\right)(i+1)$.
Now let $r$ be the least positive integer such that $\left(\bfw_\pi^{r}(f)\right)(i)\geq \left(\bfw_\pi^{r}(f)\right)(i+1)$.  Such an $r$ must exist because when repeatedly applying $\bfw_\pi$ to $f$, we must eventually cycle back around to $\bfw_\pi^{-1}(f)$.

Let $g=\bfw_\pi^{r-1}$.  Then
\begin{equation}\label{eq:melt panic}g(i)< g(i+1) \text{ and } (\bfw_\pi(g))(i) \geq (\bfw_\pi(g))(i+1).\end{equation}  By Propositions~\ref{prop:whirl-rg} and~\ref{prop:whirl-rg-n}, either $(\bfw_\pi(g))(i)=g(i)+1$ or $(\bfw_\pi(g))(i+1)=1$.

Suppose $(\bfw_\pi(g))(i+1)\not=1$. Then $(\bfw_\pi(g))(i+1)=g(i+1)$ or $(\bfw_\pi(g))(i+1)=g(i+1)+1$, the latter of which violates Equation~(\ref{eq:melt panic}) because $(\bfw_\pi(g))(i)=g(i)+1$.  So $(\bfw_\pi(g))(i+1)=g(i+1)$.  Then from Equation~(\ref{eq:melt panic}), we have $(\bfw_\pi(g))(i)=g(i)+1=g(i+1)$.  However, we apply $w_{i+1}$ after $w_i$, so $w_{i+1}$ would have to add 1 to the value $g(i+1)$ by Propositions~\ref{prop:whirl-rg} and~\ref{prop:whirl-rg-n}.  This contradicts $(\bfw_\pi(g))(i+1)=g(i+1)$.  Thus, \[(\bfw_\pi^r(f))(i+1)=(\bfw_\pi(g))(i+1)=1.\]

By Lemma~\ref{lem:sphynxinator}, $\left(\bfw_\pi^{m}(f)\right)(i) \not= 1$ for all $1\leq m\leq r$.  Therefore, given any function $f\in\eso$ satisfying $f(i)=1$ and $f(i+1)\not=1$, as we repeatedly apply $\bfw_\pi$ to $f$, we obtain a function that sends $i+1$ to 1 before we get another function that sends $i$ to 1.  This implies $\#\{f\in\eso : f(i)=1\} \leq \#\{f\in\eso : f(i+1)=1\}$.

By Lemma~\ref{lem:global-avg}, $\#\{f\in\calf : f(i)=1\} = \#\{f\in\calf : f(i+1)=1\}$.  Therefore, since $\#\{f\in\eso : f(i)=1\} \leq \#\{f\in\eso : f(i+1)=1\}$ for \emph{every} $\bfw_\pi$-orbit $\eso$ over $\calf$, we must have $\#\{f\in\eso : f(i)=1\} = \#\{f\in\eso : f(i+1)=1\}$ for every orbit $\eso$.
\end{proof}

Theorem~\ref{thm:whirl-rg} implies another homomesic statistic for the action of $\bfw_\pi$ on $\rg(n,k)$ only.

\begin{cor}\label{cor:orig-rg-conj}
Let $\pi$ a permutation of $\{2,3,\dots,n\}$.  Under the action of $\bfw_\pi$ on $\rg(n,k)$, the statistic \[f\mapsto {k\choose 2}f(2)-f(n)\] is homomesic with average $k(k-2)$.
\end{cor}

\begin{proof}
For $f\in\rg(n,k)$, $f(2)$ is always 1 or 2.  Consider a $\bfw_\pi$-orbit $\eso$ with length $\ell(\eso)$.  If $f(2)=2$ and $f(n)=k$ for every $f\in\eso$, then the total value of $f\mapsto {k\choose 2}f(2)-f(n)$ across $\eso$ is $\frac{k(k-1)}{2}2\ell(\eso) - k \ell(\eso) = k(k-2)\ell(\eso)$.

Suppose instead that there are $c$ functions $f\in\eso$ that satisfy $f(2)=1$. Then this decreases the total value of ${k\choose 2}f(2)$ across $\eso$ by ${k\choose 2} c$.

From Theorem~\ref{thm:whirl-rg}, there are also $c$ functions $f\in\eso$ satisfying $f(n)=1$.  If $f(n)\not=k$, then $n$ cannot be the least $i$ for which $f(i)=f(n)$.  Thus, when $f(n)\not=k$, $\bfw_\pi$ adds 1 to the value of $f(n)$.  When $f(n)=k$, either $(\bfw_\pi(f))(n)=1$ or $(\bfw_\pi(f))(n)=k$ via Proposition~\ref{prop:whirl-rg}.  So there are also $c$ functions satisfying $f(n)=j$ for any $j\in[k-1]$.  In relation to the original case where $f(n)=k$ for all $f$, this decreases the total value of $f(n)$ across $\eso$ by $c(1+2+3+\cdots+(k-1)) = {k\choose 2} c$.

Therefore, the total value of $f\mapsto {k\choose 2}f(2)-f(n)$ across $\eso$ is \[k(k-2)\ell(\eso) - {k\choose 2} c + {k\choose 2} c = k(k-2)\ell(\eso).\]  Hence this statistic has average $k(k-2)$ on $\eso$.
\end{proof}
\subsection{Whirling on noncrossing RG-words}
We finish the paper by considering whirling on the set of RG-words corresponding to \emph{noncrossing} partitions.

\begin{defn}
A partition of $[n]$ is said to be \dfn{noncrossing} if whenever $i<j<k<\ell$, it is \emph{not} the case that $i$ and $k$ belong to one block of $\pi$ with $j$ and $\ell$ belonging to another block.
It is well-known that the number of noncrossing partitions of $[n]$ is given by the Catalan number $C_n=\frac{1}{n+1}\binom{2n}{n}$ and that the number of noncrossing partitions of $[n]$ into $k$ blocks is given by the Narayana number $N(n,k)=\frac{1}{n}\binom{n}{k}\binom{n}{k-1}$.
A \dfn{noncrossing RG-word} is an RG-word whose corresponding partition is noncrossing.  Let $\rgnc(n,k)$ and $\rgnc(n)$ denote the sets of all noncrossing RG-words in $\rg(n,k)$ and $\rg(n)$ respectively.
\end{defn}

The proof of the following proposition is obvious.

\begin{prop}\label{prop:abab}
An RG-word $f$ is noncrossing if and only if in the one-line notation, we do not have $abab$ in order (not necessarily consecutively) for some $a\not=b$.
\end{prop}

\begin{example}
Let $f=122{\color{red}1}3{\color{red}3}{\color{red}1}4{\color{red}3}$.  Then $f\in\rg(9,4)$ but $f\not\in\rgnc(9,4)$ because the red values are $1313$.
\end{example}

Due to the noncrossing condition, there is not such a simple description (like Proposition~\ref{prop:whirl-rg} or~\ref{prop:whirl-rg-n}) of how $w_i$ acts for $\calf=\rgnc(n,k)$ or $\calf=\rgnc(n)$. In the following example, we see $w_i$ can decrease $f(i)$ to something other than 1, or increase $f(i)$ by more than one.

\begin{ex}\label{ex:RG-NC}
For $\calf=\rgnc(6,4)$,
\[
123442\stackrel{w_1}{\longmapsto}
123442\stackrel{w_2}{\longmapsto}
123442\stackrel{w_3}{\longmapsto}
123442\stackrel{w_4}{\longmapsto}
123242\stackrel{w_5}{\longmapsto}
123242\stackrel{w_6}{\longmapsto}
123244
\]
so $\bfw(123442)=123244$.

For $\calf=\rgnc(6)$,
\[
123442\stackrel{w_1}{\longmapsto}
123442\stackrel{w_2}{\longmapsto}
123442\stackrel{w_3}{\longmapsto}
123442\stackrel{w_4}{\longmapsto}
123242\stackrel{w_5}{\longmapsto}
123222\stackrel{w_6}{\longmapsto}
123224
\]
so $\bfw(123442)=123224$.
\end{ex}

\begin{thm}\label{thm:rgnc}
Let $\calf=\rgnc(n,k)$ for some $1\leq k\leq n$.  Under the action of $\bfw$ on $\calf$, $I_{2\mapsto 1}-I_{n\mapsto 1}$ is $0$-mesic.
\end{thm}

We also conjecture that this homomesy holds for other whirling orders, and on $\calf=\rg(n)$.

\begin{conj}\label{conj:rgnc}
Fix $\calf$ to be either $\rgnc(n)$ or $\rgnc(n,k)$ for some $1\leq k\leq n$.  Let $\pi$ be a permutation of $\{2,3,\dots,n\}$.  Under the action of $\bfw_\pi$ on $\calf$, $I_{2\mapsto 1}-I_{n\mapsto 1}$ is $0$-mesic.
\end{conj}

We utilize some lemmas in the proof of Theorem~\ref{thm:rgnc}.

\begin{lemma}\label{lem:aab->abb}
Let $f\in\calf=\rgnc(n,k)$ for some $1\leq k\leq n$, and let $1<i<n$.  If $f(i)=f(i-1)$, then $(w_i(f))(i)=f(i+1)$.
\end{lemma}

\begin{proof}
Let $a=f(i-1)=f(i)$ and $b=f(i+1)$.  Clearly we can change the value of $f(i)$ from $a$ to $b$, and it will remain a noncrossing RG-word.  We will now explain why changing the value of $f(i)$ to the various other values $c$ that we encounter on the way from $a$ to $b$ when we apply $w_i$ will not yield have a noncrossing RG-word.\\
\textbf{Case 1: $a=b$.} Since $f\in\rgnc(n,k)$, there must be some $j$ for which $f(j)=c$.  So changing the value of $f(i)$ to $c$ will not yield a noncrossing RG-word by Proposition~\ref{prop:abab}.\\
\textbf{Case 2: $a<c<b$.} Since $f$ is an RG-word, there must exist $j<h<i-1$ such that $f(j)=a$ and $f(h)=c$.  So changing the value of $f(i)$ to $c$ will not yield a noncrossing RG-word by Proposition~\ref{prop:abab}.\\
\textbf{Case 3: $b<a<c$.} Since $f\in\rgnc(n,k)$, there must be some $j$ for which $f(j)=c$.  If $j<i-1$, then since $f$ is an RG-word, there must exist $h<j$ for which $f(h)=a$.  Then changing the value of $f(i)$ to $c$ will not yield a noncrossing RG-word by Proposition~\ref{prop:abab}.  On the other hand, suppose $j>i+1$.  Since $f$ is an RG-word, there must exist $h<i-1$ for which $f(h)=b$, in which case changing the value of $f(i)$ to $c$ will not yield a noncrossing RG-word by Proposition~\ref{prop:abab}.\\
\textbf{Case 4: $c<b<a$.}  Since $f$ is an RG-word, we must have $h<j<i-1$ satisfying $f(h)=c$ and $f(j)=b$.  So changing the value of $f(i)$ to $c$ will not yield a noncrossing RG-word by Proposition~\ref{prop:abab}.
\end{proof}

\begin{lemma}\label{lem:aa->a1}
Let $f\in \calf=\rgnc(n,k)$ for some $1\leq k\leq n$.
If $f(n-1)=f(n)$, then $(w_n(f))(n)=1$.
\end{lemma}

\begin{proof}
Let $a=f(n-1)$.  Let $a<b\leq k$.  Since $f\in\rgnc(n,k)$, there must exist $i<j<n-1$ satisfying $f(i)=a$ and $f(j)=b$.  So we cannot change the value of $f(n)$ to $b$ and still have a noncrossing RG-word by Proposition~\ref{prop:abab}.  However, changing the value of $f(n)$ to 1 will still yield a noncrossing RG-word.  So $w_n$ changes the value of $f(n)$ to 1.
\end{proof}

\begin{lemma}\label{lem:cycrot}
Let $\calf=\rgnc(n,k)$ for some $1\leq k\leq n$.
Let $f\in\calf$ for which $f(2)=1$.  Then $(\bfw(f))(n)=1$ and $(\bfw(f))(i)=f(i+1)$ for all $2\leq i< n$.
\end{lemma}

The proof of the above lemma follows easily by Lemmas~\ref{lem:aab->abb} and~\ref{lem:aa->a1}.  We are now ready to prove Theorem~\ref{thm:rgnc}.

\begin{proof}[of Theorem~\ref{thm:rgnc}]
Note that among noncrossing partitions of $[n]$ into $k$ blocks:
\begin{itemize}
\item there are $N(n-1,k)$ of them that have 1 and 2 in the same block, and
\item there are $N(n-1,k)$ of them that have 1 and $n$ in the same block.
\end{itemize}
The first is because we can take any noncrossing partition of $\{2,3,\ldots,n\}$ and place 1 in the same block as 2 and still have a noncrossing partition.  Likewise, the second is because we can take any noncrossing partition of $[n-1]$ and place $n$ in the same block as 1 and still have a noncrossing partition.  So $\calf$ contains the same number of functions $f$ satisfying $f(2)=1$ as functions $f$ satisfying $f(n)=1$.
By Lemma~\ref{lem:cycrot}, $\bfw$ is a bijection between functions $f$ satisfying $f(2)=1$ and functions $f$ satisfying $f(n)=1$.  So each $\bfw$-orbit contains equal numbers of these two types of functions, and thus $I_{2\mapsto 1}-I_{n\mapsto 1}$ is 0-mesic.
\end{proof}

\acknowledgements
\label{sec:ack}
We have benefitted from useful discussions with Yue Cai, David Einstein, Max Glick, Shahrzad Haddadan, Sam Hopkins, Thomas McConville, Matthew Plante, Elizabeth Sheridan Rossi, Richard Stanley, Jessica Striker, and Nathan Williams.  Computations leading to the initial conjectures were done in Sage (\cite{sage}).  We also thank two anonymous reviewers who have provided helpful feedback.

\nocite{*}
\bibliographystyle{abbrvnat}
\bibliography{bibliography}
\label{sec:biblio}

\end{document}